\documentclass{amsart}

\usepackage{graphicx,amssymb,latexsym,amsfonts,txfonts,amsmath,amsthm}
\usepackage{pdfsync,color,enumitem}

\usepackage{hyperref}
\hypersetup{
    colorlinks=true,       
    linkcolor=blue,          
    citecolor=blue,        
    filecolor=blue,      
    urlcolor=blue           
}

\newtheorem{theo}{Theorem}
\newtheorem{coro}{Corollary}
\newtheorem{propo}{Proposition}
\newtheorem{lemm}{Lemma}

\theoremstyle{remark}
\newtheorem{rema}{\bf Remark}

\begin{document}

\title{A simple remark on holomorphic maps on Torelli space of marked spheres}

\author{Ruben A. Hidalgo}
\address{Departamento de Matem\'atica y Estad\'{\i}stica, Universidad de La Frontera. Temuco, Chile}
\email{ruben.hidalgo@ufrontera.cl}

\thanks{Partially supported by Project Fondecyt 1230001}

\maketitle


\begin{abstract}
The configuration space of $k \geq 3$ ordered points in the Riemann sphere $\widehat{\mathbb C}$ is the 
Torelli space ${\mathcal U}_{0,k}$; a complex manifold of dimension $k-3$. 
If $m,n \geq 4$ and
$F:{\mathcal U}_{0,m} \to {\mathcal U}_{0,n}$ is a non-constant holomorphic map, then we observe that (i) $n \leq m$ and (ii) each coordinate of $F$ is given by a cross-ratio.
\end{abstract}

\section{Introduction}
The Torelli space ${\mathcal U}_{g,m+1}$, where 
$g \geq 0, m \geq -1$ are integers such that $3g-2+m>0$, 
 is the one that parametrizes pairs $(S,(p_{1},\ldots,p_{m+1}))$, where $S$ is a closed Riemann surface of genus $g$  and $p_{1}, \ldots, p_{m+1} \in S$ are pairwise different points, up to biholomorphisms (two such pairs $(S_{1},(p_{1},\ldots,p_{m+1}))$ and  $(S_{2},(q_{1},\ldots,q_{m+1}))$ are biholomorphically equivalents if there is a biholomorphism $f:S_{1} \to S_{2}$ such that $f(p_{j})=q_{j}$, for every $j$). This space is a complex orbifold of dimension $3g-2+m$ \cite{Imayoshi,Nag}. Note that ${\mathcal U}_{g,0}={\mathcal M}_{g}$ is the moduli space of closed Riemann surfaces of genus $g$, and ${\mathcal U}_{0,3}$ is just one point.
 
In \cite{AAS}, Antonakoudis-Aramayona-Souto have proved that,
if $g \geq 6$ and $g' \leq 2g-2$, then every non-constant holomorphic map ${\mathcal U}_{g,m+1} \to {\mathcal U}_{g',n+1}$ is a forgetful map (in particular, $g'=g$ and $n \leq m$). In the same paper, it is also observed that, for $g \geq 4$, every non-constant holomorphic map ${\mathcal U}_{g,m+1} \to {\mathcal U}_{g,m+1}$ is induced by a permutation of the marked points. In \cite{ALS}, Aramayona-Leiniger-Souto proved that for every $g \geq 2$ there is some $g'>g$ and an holomorphic embedding ${\mathcal M}_{g}={\mathcal U}_{g,0} \hookrightarrow {\mathcal M}_{g'}={\mathcal U}_{g',0}$. 
In this paper, we consider the genus zero situation, that is, we study (non-constant) holomorphic maps ${\mathcal U}_{0,m+1} \to {\mathcal U}_{0,n+1}$, where $m,n \geq 3$. 

Let us first observe that, if $k \geq 3$, then the Torelli space ${\mathcal U}_{0,k+1}$ (of dimension $k-2$) can be naturally identified with the domain 
$\Omega_{k}=\{(z_{1},\ldots,z_{k-2}): z_{j} \in {\mathbb C}\setminus\{0,1\}, z_{i} \neq z_{j}, \; i\neq j\} \subset {\mathbb C}^{k-2}$ and that its group of holomorphic automorphisms is ${\mathbb G}_{k}=\langle A, B \rangle$, where
$$A(z_{1},\ldots,z_{k-2})=(z_{1}^{-1},\ldots,z_{k-2}^{-1}),$$
$$B(z_{1},\ldots,z_{k-2})=\left(\frac{z_{k-2}}{z_{k-2}-1},\frac{z_{k-2}}{z_{k-2}-z_{1}},\ldots,\frac{z_{k-2}}{z_{k-2}-z_{k-3}}\right).$$

If $k=3$, then ${\mathcal U}_{0,4}=\Omega_{3}=\widehat{\mathbb C} \setminus \{\infty,0,1\}$ and 
${\mathbb G}_{3} \cong {\mathfrak S}_{3}$. In this case, every non-constant holomorphic map $F:\Omega_{3} \to \Omega_{3}$ is an automorphism (so given as one of the six possible cross ratios of the points $\{\infty,0,1,z\}$, for $z \in \Omega_{3}$). If $k \geq 4$, then each element of ${\mathbb G}_{k} \cong {\mathfrak S}_{k+1}$ can be described in terms of cross-ratios using four points inside the set $\{\infty,0,1,z_{1},\ldots,z_{k-2}\}$ (see Section \ref{Sec:explicito}).

As an application of the Picard theorem for Riemann surfaces (due to Royden \cite{Royden}), we can obtain the following.

\begin{theo}\label{main}
Let $F:\Omega_{m} \to \Omega_{n}$, where $n,m \geq 3$ are integers, a non-constant holomorphic map. Then  
\begin{enumerate}[leftmargin=15pt]
\item[(a)] $n \leq m$, and 
\item[(b)] there exist:
\begin{enumerate}
\item[(i)] a subset $\{j_{1},\ldots,j_{n-2}\} \subset \{1,\ldots,m-2\}$, where $j_{i} \neq j_{l}$ for $i \neq l$, and 
\item[(ii)] some $T=(T_{1},\ldots,T_{m-2}) \in {\mathbb G}_{m}$,
\end{enumerate}
such that $F=(T_{j_{1}},\ldots,T_{j_{n-2}})$.
\end{enumerate}
\end{theo}

The result above shows that any non-constant holomorphic map between Torelli spaces of genus zero and at least 4 points can be expressed as a composition of an automorphism
$\Omega_{m} \to \Omega_{m}$ and a forgetful map $\Omega_{m} \to \Omega_{n}$, with $n \leq m$.
The above is related to \cite[Conjecture 1.2]{GeGre} at the level of Teichm\"uller spaces.

Below, we consider the usual cross-ratio $[a,b,c,d]:=(d-b)(c-a)/(d-a)(c-b)$. As already mentioned above, every non-constant holomorphic map $F:\Omega_{3} \to \Omega_{3}$ is an automorphism. The following provides a generalization.

\begin{coro}\label{coro1}
If $F:\Omega_{m} \to \Omega_{3}$ is a non-constant holomorphic map, where $m \geq 3$, then there is an ordered tuple $(i_{1},i_{2},i_{3},i_{4}) \in \{1,\ldots,m+1\}^{4}$, where $i_{j} \neq i_{k}$ for $i \neq k$, such that 
$$F(z_{1},\ldots,z_{m-2})=[p_{i_{1}}(z),p_{i_{2}}(z),p_{i_{3}}(z),p_{i_{4}}(z)],$$ where 
$p_{1}(z)=\infty$, $p_{2}(z)=0$, $p_{3}(z)=1$, $p_{4}(z)=z_{1}, \ldots, p_{m+1}(z)=z_{m-2}$. 
\end{coro}

\section{Preliminaries}\label{Sec:prelim}

\subsection{The Torelli's space ${\mathcal U}_{0,k+1}$}
Let $k \geq 3$ be an integer and let us fix $k+1$ different points $p_{1},\ldots, p_{k+1}$ on the two-dimensional sphere $S^{2}$. Two orientation-preserving homeomorphisms $\phi_{1}, \phi_{2}:S^{2} \to \widehat{\mathbb C}$ are called equivalent if there exists a M\"obius transformation $A \in {\rm PSL}_{2}({\mathbb C})$ such that $A(\phi_{1}(p_{j}))=\phi_{2}(p_{j})$, for every $j=1,\ldots,k+1$, and $\phi_{2}^{-1} \circ A \circ \phi_{1}$ is homotopic to the identity relative the set $\{p_{1},\ldots,p_{k+1}\}$. The space of the above equivalence classes is the Teichm\"uller space ${\mathcal T}_{0,k+1}$.; which is known to be a $k-2$ dimensional complex manifold, homeomorphic to a ball (see, for instance, \cite{Hubbard,Imayoshi}).  

Let ${\rm Hom}^{+}(\{p_{1},\ldots,p_{k+1}\})$ be the group of orientation-preserving homeomorphisms of $S^{2}$ keeping invariant the set $\{p_{1},\ldots,p_{k+1}\}$. Let ${\rm Hom}^{+}((p_{1},\ldots,p_{k+1}))$ be its subgroup of those homeomorphisms fixing each of the points $p_{j}$, and let ${\rm Hom}_{0}((p_{1},\ldots,p_{k+1}))$ be its (normal) subgroup of those being homotopic to the identity relative to the set $\{p_{1},\ldots,p_{k+1}\}$. The quotient group ${\rm Mod}_{0,k+1}={\rm Hom}^{+}(\{p_{1},\ldots,p_{k+1}\})/{\rm Hom}_{0}((p_{1},\ldots,p_{k+1}))$ 
(respectively, ${\rm PMod}_{0,k+1}={\rm Hom}^{+}((p_{1},\ldots,p_{k+1}))/{\rm Hom}_{0}((p_{1},\ldots,p_{k+1}))$)
is called the 
modular group (respectively, the pure modular group). 

\begin{rema}
If ${\rm Aut}({\mathcal T}_{0,k+1})$ denotes the group of holomorphic automorphisms of ${\mathcal T}_{0,k+1}$, then 
there is a natural surjective homomorphism $\Theta:{\rm Mod}_{0,k+1} \to {\rm Aut}({\mathcal T}_{0,k+1})$ \cite{E-K}.
If, moreover, $k \geq 4$, then $\Theta$ is an isomorphism \cite{E-G, Epstein, Matsuzaki}.
 \end{rema}
 
The modular group ${\rm Mod}_{0,k+1}$ acts discontinuously on ${\mathcal T}_{0,k+1}$ as group of holomorphic automorphisms and the quotient orbifold ${\mathcal M}_{0,k+1}={\mathcal T}_{0,k+1}/{\rm Mod}_{0,k+1}$ is the moduli space that parametrizes the conformal structures, up to biholomorphisms, of unordered $(k+1)$-marked spheres. This space is a $(k-2)$-dimensional complex orbifold whose orbifold fundamental group is ${\rm Mod}_{0,k+1}$ (see, for instance, \cite{Bers1,Bers2,Hubbard}). 

The pure modular group ${\rm PMod}_{0,k+1}$ is known to act freely on ${\mathcal T}_{0,k+1}$, so the quotient space ${\mathcal U}_{0,k+1}={\mathcal T}_{0,k+1}/{\rm PMod}_{0,k+1}$, called the Torelli space of $(k+1)$-marked spheres, is a complex manifold of dimension $k-2$. This space parametrizes ${\rm PSL}_{2}({\mathbb C})$-equivalence classes of ordered $(k+1)$-points in $\widehat{\mathbb C}$. The group of holomorphic automorphisms of ${\mathcal U}_{0,k+1}$ is given by the quotient group 
${\rm Aut}({\mathcal U}_{0,k+1})={\rm Mod}_{0,k+1}/{\rm PMod}_{0,k+1}$.

\subsection{An explicit model of ${\mathcal U}_{0,k+1}$ and its automorphisms}\label{Sec:explicito}
As the group of M\"obius transformations acts triple-transitive, for each tuple $(q_{1},\ldots,q_{k+1}) \in \widehat{\mathbb C}^{k+1}$ there is a unique M\"obius transformation $T$ such that  
$$(T(q_{1})=\infty, T(q_{2})=0, T(q_{3})=1, T(q_{4})=z_{1},\ldots, T(q_{k+1})=z_{k-2}).$$

It follows that a model for ${\mathcal U}_{0,k+1}$, $k \geq 3$, is given by the domain \cite{Patterson}
$$\Omega_{k}=\{(z_{1},\ldots,z_{k-2}): z_{j} \in {\mathbb C}\setminus\{0,1\}, z_{i} \neq z_{j}, \; i\neq j\} \subset {\mathbb C}^{k-2}.$$

Let us set ${\mathbb G}_{k}:={\rm Aut}(\Omega_{k}) \cong {\rm Aut}({\mathcal U}_{0,k+1})$.

If $k=3$, then ${\mathcal U}_{3}=\widehat{\mathbb C} \setminus \{\infty,0,1\}$ and ${\mathbb G}_{3}={\rm Aut}(\Omega_{3})=\langle A(z)=z^{-1}, B(z)=\frac{z}{z-1}\rangle$.

\begin{rema}
If $k=3$, then the quotient group ${\rm Mod}_{0,4}/{\rm PMod}_{0,4} \cong {\mathfrak S}_{4}$ has a normal subgroup $K \cong {\mathbb Z}_{2}^{2}$ acting trivially. Also, ${\mathcal M}_{0,4}=\Omega_{3}/{\mathbb G}_{3}={\mathbb C}$ (with two cone points, one of order $2$ and the other of order $3$). 
\end{rema}

\begin{lemm}
If $k \geq 4$, then 
${\mathbb G}_{k}=\langle A,B\rangle \cong {\mathfrak S}_{k+1}$, 
where $$A(z_{1},\ldots,z_{k-2})=(z_{1}^{-1},\ldots,z_{k-2}^{-1}),$$
$$B(z_{1},\ldots,z_{k-2})=\left(\frac{z_{k-2}}{z_{k-2}-1},\frac{z_{k-2}}{z_{k-2}-z_{1}},\ldots,\frac{z_{k-2}}{z_{k-2}-z_{k-3}}\right).$$
\end{lemm}
\begin{proof}
It is not difficult to check that $A$ and $B$, as described in the lemma, are holomorphic automorphisms of $\Omega_{k}$.
To see the above, one first observe that ${\mathbb G}_{k} \leq {\rm Aut}(\Omega_{k})$. Secondly, as ${\rm Aut}({\mathcal U}_{0,k+1})={\rm Mod}_{0,k+1}/{\rm PMod}_{0,k+1}$ is a subgroup of ${\mathfrak S}_{k+1}$, we only need to check that ${\mathbb G}_{k} \cong {\mathfrak S}_{k+1}$. 
Next, we provide an explicit isomorphism 
$$\Theta_{k}:{\mathfrak S}_{k+1} \to {\mathbb G}_{k}: \sigma \mapsto (T_{1}^{\sigma},\ldots,T_{k-2}^{\sigma}).$$

Let $\sigma \in {\mathfrak S}_{k+1}$. If $(z_{1},\ldots,z_{k-2}) \in \Omega_{k}$, then
consider the two tuples 
$$P=(p_{1}=\infty, p_{2}=0, p_{3}=1, p_{4}=z_{1}, \ldots, p_{k+1}=z_{k-2}),$$ 
$$P_{\sigma}=(p_{\sigma^{-1}(1)}, p_{\sigma^{-1}(2)}, p_{\sigma^{-1}(3)}, p_{\sigma^{-1}(4)}, \ldots, p_{\sigma^{-1}(k+1)}).$$

There is a unique M\"obius transformation $M_{\sigma}$ such that 
$$M_{\sigma}(p_{\sigma^{-1}(1)})=\infty, \; M_{\sigma}(p_{\sigma^{-1}(2)})=0, \; M_{\sigma}(p_{\sigma^{-1}(3)})=1,$$
which is given explicitly as the cross-ratio
$$M_{\sigma}(x)=[p_{\sigma^{-1}(1)},p_{\sigma^{-1}(2)},p_{\sigma^{-1}(3)},x].$$

If we apply $M_{\sigma}$ to each of the coordinates of the tuple $P_{\sigma}$ we obtain the new tuple
$$(\infty,0,1,M_{\sigma}(p_{\sigma^{-1}(4)}), \ldots, M_{\sigma}(p_{\sigma^{-1}(k+1)})).$$

It can be seen that $(M_{\sigma}(p_{\sigma^{-1}(4)}), \ldots, M_{\sigma}(p_{\sigma^{-1}(k+1)})) \in \Omega_{k}$. 

Moreover, if  $T_{j}^{\sigma}(z_{1},\ldots,z_{k-2}):=M_{\sigma}(p_{\sigma^{-1}(3+j)})$, then 
$T_{j}^{\sigma}:\Omega_{k} \to \Omega_{3}$ is surjective and extends to a holomorphic map 
$T_{j}^{\sigma}:{\mathbb C}^{k-2} \to \widehat{\mathbb C}$
such that $$(T_{j}^{\sigma})^{-1}(\{\infty,0,1\})=\bigcup_{j=1}^{k-2} \left(\{z_{j}=0\} \cup\{z_{j}=1\}\right) \bigcup_{i\neq j}\{z_{i}=z_{j}\}.$$

Now, it is not difficult to see that $\Theta_{k}(\sigma)=(T_{1}^{\sigma},\ldots,T_{k-2}^{\sigma}) \in {\mathbb G}_{k}$. 
\end{proof}

\subsection{\bf Forms of the coordinates $T_{j}^{\sigma}$ of $\Theta_{k}$}\label{Sec:formas}
In the above proof, we have explicitly constructed $\Theta(\sigma)=T^{\sigma}$. Below, we write down the explicit forms of its 
coordinates.

\subsubsection*{\rm(1)}
If $k=4$, then the values of $T_{j}^{\sigma}(z_{1},z_{2})$ ($j=1,2$) are as follows ($s,j \in \{1,2\}$, $s \neq l$):
$$z_{s}, 1/z_{s},\; 1-z_{s},\; 1/(1-z_{s}),\; z_{s}/(z_{s}-1),\; (z_{s}-1)/z_{s}, 
z_{s}/z_{l},\; (z_{s}-z_{l})/z_{s}, \; z_{s}/(z_{s}-z_{l}),$$
$$(z_{s}-1)/(z_{l}-1),\; (z_{s}-z_{l})/(z_{s}-1),\; (z_{s}-1)/(z_{s}-z_{l}), \; z_{l}(z_{s}-1)/(z_{s}-z_{l}),\; (z_{s}-z_{l})/(z_{l}(z_{s}-1),$$
$$z_{s}(z_{l}-1)/z_{l}(z_{s}-1).$$

\subsubsection*{\rm(2)}
If $k=5$, then $T_{j}^{\sigma}(z_{1},z_{2},z_{3})$ ($j=1,2,3$) are as follows ($s,l,i \in \{1,2,3\}$ pairwise different):
$$(z_{i}-z_{s})/z_{s}(z_{i}-1),\; z_{i}(z_{s}-1)/(z_{s}-z_{i}),\; z_{i}(z_{s}-1)/z_{s}(z_{i}-1), \; z_{i}(z_{l}-z_{s})/z_{l}(z_{i}-z_{s}),$$ 
$$(z_{i}-1)(z_{l}-z_{s})/(z_{l}-1)(z_{i}-z_{s}),\; (z_{l}-z_{s})/(z_{i}-z_{s}),\; z_{i}/(z_{i}-z_{s}),\; (z_{i}-z_{s})/z_{i},\; z_{i}/z_{s},$$
$$(z_{i}-1)/(z_{i}-z_{s}),\; (z_{i}-z_{s})/(z_{i}-1),\; (z_{i}-1)/(z_{s}-1).$$

\subsubsection*{\rm(3)}
If $k \geq 6$, then $T_{j}^{\sigma}(z_{1},\ldots,z_{k-2})$ ($j=1,\ldots,k-2$) are as follows ($s,l,i,r \in \{1,\ldots,k-2\}$ pairwise different):
$$(z_{s}-z_{i})(z_{r}-z_{l})/(z_{s}-z_{l})(z_{r}-z_{i}), 
z_{i}(z_{r}-z_{l})/z_{l}(z_{r}-z_{i}), \; (z_{i}-1)(z_{r}-z_{l})/(z_{l}-1)(z_{r}-z_{i}),$$
$$(z_{r}-z_{l})/(z_{r}-z_{i}),\; z_{i}/z_{r},\; z_{i}/(z_{i}-z_{r}),\; (z_{i}-z_{r})/z_{i}, 
(z_{i}-1)/(z_{r}-1),\; (z_{i}-1)/(z_{i}-z_{r}),\; (z_{i}-z_{r})/(z_{i}-1).$$

\begin{rema}
(1) In \cite{Igusa}, by using invariant theory, Igusa  observed that 
$$
\begin{array}{l}
\Omega_{4}/{\mathbb G}_{4}={\mathcal M}_{0,5}={\mathbb C}^{2}/\langle (x,y) \mapsto (-x,-y)\rangle,\\
\Omega_{5}/{\mathbb G}_{5}={\mathcal M}_{0,6}={\mathbb C}^{3}/\langle (x,y,z) \mapsto (\omega_{5} x, \omega_{5}^{2} y, \omega_{5}^{3} z)\rangle \quad (\omega_{5}=e^{2 \pi i/5}).
\end{array}
$$ 
(2) In \cite{Patterson}, Patterson proved that, for $k \geq 6$, the moduli space $\Omega_{k}/{\mathbb G}_{k}={\mathcal M}_{0,k+1}$ cannot be obtained as a quotient ${\mathbb C}^{k-2}/L$, where $L$ is a finite linear group.
\end{rema}

\subsection{Auxiliary facts}
\subsubsection{}
There is a natural permutation action of ${\mathfrak S}_{k-2}$ on $\Omega_{k}$, given by the permutation of the coordinates. This induces an embedding of ${\mathfrak S}_{k-2}$ as a subgroup of ${\mathbb G}_{k}$. Below, we proceed to describe it.

\begin{lemm}
If $T=(T_{1},\ldots,T_{k-2}) \in {\mathbb G}_{k}$ and $\tau \in {\mathfrak S}_{k-2}$, then $(T_{\tau(1)},\ldots,T_{\tau(k-2)}) \in {\mathbb G}_{k}$.
\end{lemm}
\begin{proof}
If $\sigma=(1)(2)(3)\widehat{\tau} \in {\mathfrak S}_{k+1}$, where $\widehat{\tau}(j+3)=\tau(j)$, then $\Theta_{k}(\sigma)(z_{1},\ldots,z_{k-2})=(z_{\tau(1)},\ldots, z_{\tau(k-2)})$ and
$\Theta_{k}(\sigma) \circ T \in {\mathbb G}_{k}$.
\end{proof}

\subsubsection{}
For $k \geq 4$, we set $I_{k}=\{1,2,\ldots,k+1\}$ and $A_{k}=I_{k}^{4} \setminus {\rm diagonals}$. For each tuple $C:=(i_{1},i_{2},i_{3},i_{4}) \in A_{k}$ we set
$$L_{C}:\Omega_{k} \to \Omega_{3}: z=(z_{1},\ldots,z_{k-2}) \mapsto L_{C}(z)=[p_{i_{1}}(z),p_{i_{2}}(z),p_{i_{3}}(z),p_{i_{4}}(z)],$$
where
$$p_{1}(z)=\infty, p_{2}(z)=0, p_{3}(z)=1, p_{4}(z)=z_{1}, \ldots, p_{k+1}(z)=z_{k-2}.$$

\begin{propo}\label{prop1}
Let $k \geq 4$ and let $C_{1},\ldots,C_{l} \in A_{k}$. Then
\begin{enumerate}[leftmargin=15pt]
\item If $(L_{C_{1}},\ldots,L_{C_{l}}):\Omega_{k} \to \Omega_{l+2}$ is a non-constant holomorphic map, then $l \leq k-2$.

\item If $l \leq k-2$ and $(L_{C_{1}},\ldots,L_{C_{l}}):\Omega_{k} \to \Omega_{l+2}$ is a non-constant holomorphic map, then there exists $T \in {\mathbb G}_{k}$ such that $T=(L_{C_{1}},\ldots,L_{C_{l}},T_{l+1},\ldots,T_{k-2})$.

\end{enumerate}
\end{propo}
\begin{proof}
The result follows from the following observation.
If $C_{1}=(i_{1},i_{2},i_{3},i_{4}), C_{2}=(j_{1},j_{2},j_{3},j_{4}) \in A_{k}$, then the equation $L_{C_{1}}(z)=L_{C_{2}}(z)$ has no solution for $z \in \Omega_{k}$ if and only if we are in any of the following four situations:
$$(a) \; (i_{1},i_{2},i_{3})=(j_{1},j_{2},j_{3}), \; i_{4} \neq j_{4}, \quad (b) \; (i_{1},i_{2},i_{4})=(j_{1},j_{2},j_{4}), \; i_{3} \neq j_{3},$$
$$(c) \; (i_{1},i_{3},i_{4})=(j_{1},j_{3},j_{4}), \; i_{2} \neq j_{2}, \quad (d) \; (i_{2},i_{3},i_{4})=(j_{2},j_{3},j_{4}), \; i_{1} \neq j_{1}.$$
\end{proof}

\subsubsection{}
Let us assume $3 \leq n \leq m$. If $\{j_{1},\ldots,j_{n-2}\} \subset \{1,\ldots,m-2\}$, then we set the projection map
$$\pi_{(j_{1},\ldots,j_{n-2})}:\Omega_{m} \to \Omega_{n}:(z_{1},\ldots,z_{m-2}) \mapsto (z_{j_{1}},\ldots, z_{j_{n-2}}).$$

\begin{propo}
If $T \in {\mathbb G}_{n}$, then there exists $U \in {\mathbb G}_{m}$ such that 
$$\pi_{(j_{1},\ldots,j_{n-2})} \circ U= T \circ \pi_{(j_{1},\ldots,j_{n-2})}$$
\end{propo}
\begin{proof}
Let $T=\Theta_{n}(\sigma)$, for $\sigma \in {\mathfrak S}_{n+1}$. The idea is to use $\sigma$ to construct $\widehat{\sigma} \in {\mathfrak S}_{m+1}$ such that $U=\Theta_{m}(\widehat{\sigma})$. This is done as follows. 
Set $\widehat{\sigma}$ to fix all the points in $\{1,\ldots,m+1\}\setminus\{1,2,3,j_{1}+3,\ldots,j_{n-2}+3\}$, for $r=1,2,3$, we set 
$$\widehat{\sigma}(r)=\left\{\begin{array}{ll}
\sigma(r), & \mbox{if $\sigma(r)\in \{1,2,3\}$}\\
\sigma(j_{\sigma(r)-3})+3, &\mbox{if $\sigma(r)\notin \{1,2,3\}$} 
\end{array}
\right.
 $$
$$\widehat{\sigma}(j_{k}+3)=\left\{\begin{array}{ll}
\sigma(k+3), & \mbox{if $\sigma(k+3)\in \{1,2,3\}$}\\
\sigma(j_{\sigma(k+3)-3})+3, &\mbox{if $\sigma(k+3)\notin \{1,2,3\}$} 
\end{array}
\right.
 $$
\end{proof}

\subsubsection{Example}
In the above, let $m=5$, $n=4$ and $\{j_{1}=1,j_{2}=2\}$. In this case, the projection map is
$\pi_{(1,2)}:\Omega_{5} \to \Omega_{4}:(z_{1},z_{2},z_{3}) \mapsto (z_{1},z_{2})$. If $\sigma=(1)(2,3)(4,5)$ and $T=\Theta_{4}(\sigma)$, that is, $T(z_{1},z_{2})=(1-z_{2},1-z_{1})$, then $\widehat{\sigma}=(1)(2,3)(4,5)(6)$ and $U(z_{1},z_{2},z_{3})=(1-z_{2},1-z_{1},1-z_{3})$.

\subsection{The Picard Theorem for Riemann surfaces}
The proof of Theorem \ref{main} will be a consequence of the following.

\begin{theo}[The Picard theorem for Riemann surfaces \cite{Royden}]
Let $W$ be a hyperbolic Riemann surface (i.e., its universal cover is the hyperbolic plane) and let $f:\{0<|z|<1\} \to W$ be a holomorphic map. Then $f$ can be extended to a holomorphic map $\widetilde{f}:\{|z|<1\} \to W^{*}$, where $W^{*}$ is a Riemann surface containing $W$ (if $W^{*} \neq W$, then $W^{*}=W \cup \{p\}$).
\end{theo}

\section{Proof of Theorem \ref{main}}

\subsection{Case $m=n=3$}
Let $F:\Omega_{3} \to \Omega_{3}$ be a 
non-constant holomorphic map. 
By the Picard theorem for Riemann surfaces, $F$ extends to a holomorphic non-constant map $\widehat{F}:\widehat{\mathbb C} \to \widehat{\mathbb C}$, that is, a rational map. As $\widehat{F}^{-1}(\{\infty,0,1\})=\{\infty,0,1\}$ and $\widehat{F}$ is surjective, up to composition with an element of ${\mathbb G}_{3}$, we may assume that $\widehat{F}$ fix each of the points $\infty, 0,1$. It follows that $\widehat{F}$ is the identity. It follows that $F \in {\mathbb G}_{3}$.

\subsection{Case $m=3$ and $n \geq 4$}
Let $F=(F_{1},\ldots,F_{n-2}):\Omega_{3} \to \Omega_{n}$, $n \geq 4$, be a non-constatnt holomorphic map. Then each $F_{j}:\Omega_{3} \to \Omega_{3}$ is holomorphic. Let us assume that $F$ is non-constant. Then one of the coordinates $F_{j}$ is non-constant and, by the above-proved case $n=m=3$, it belongs to ${\mathbb G}_{3}$. Up to post-composition by a suitable element of ${\mathbb G}_{3}$, we may assume this is the case for $F_{1}$. We claim that every $F_{j}$ is non-constant. In fact, assume, for instance that $F_{j}$ (for some $j \in \{2,\ldots,n-2\}$) is a constant $\alpha \in \Omega_{3}$. As $F_{1}$ is surjective, there exists some $p \in \Omega_{3}$ such that $F_{1}(p)=\alpha=F_{j}(p)$. This asserts that $F(p)$ cannot belong to $\Omega_{n}$ as it will have two different coordinates equal to $\alpha$. Now, by the above-proved case $n=m=3$, all coordinates $F_{j} \in {\mathbb G}_{3}$, and they are pairwise different (for $F$ to have the image in $\Omega_{n}$). But, in this case, $F_{j}^{-1} \circ F_{1} \in {\mathbb G}_{3}\setminus\{I\}$, for $j=2,\ldots,n-2$. As each non-trivial element of ${\mathbb G}_{3}$ has fixed points in $\Omega_{3}$, it follows the existence of a point $p_{j} \in \Omega_{3}$ with $F_{j}(p_{j})=F_{1}(p_{j})$, which is a contradiction to the fact that $F(p_{j})\in \Omega_{n}$.

\subsection{Case $n=3$ and $m \geq 4$}
Let $F:\Omega_{m} \to \Omega_{3}$ be a non-constant holomorphic map, for $m \geq 4$. It follows that $F$ is non-constant at some of the variables $z_{j}$, $j=1,\ldots,m-2$. Without loss of generality, we may assume $F$ is non-constant in the variable $z_{1}$. For each $(z_{2},\ldots,z_{m-2}) \in \Omega_{m-1}$, if we set $\Omega(z_{2},\ldots,z_{m-2}):=\widehat{\mathbb C}\setminus \{\infty, 0,1, z_{2},\ldots, z_{m-2}\}$, then 
$$\Omega_{m}=\bigcup_{(z_{2},\ldots,z_{m-2})\in \Omega_{m-1}} \Omega(z_{2},\ldots,z_{m-2}).$$

Let us consider the restrictions $F:\Omega(z_{2},\ldots,z_{m-2}) \to \Omega_{3}$, which are non-constant holomorphic maps. Applying the Picard theorem for Riemann surfaces, we have that the restriction $F$ is the restriction of some rational map (in the variable $z_{1}$), say $\widehat{F}:\widehat{\mathbb C} \to \widehat{\mathbb C}$.
We also have that
$$\widehat{F}(\{\infty,0,1,z_{2},\ldots,z_{m-2}\})=\{\infty,0,1\}$$
$$\widehat{F}^{-1}(\{\infty,0,1\}) \subset \{\infty,0,1,z_{2},\ldots,z_{m-2}\}$$

In this way, the holomorphic map $F:\Omega_{m} \to \Omega_{3}$ is the restriction of a rational map
$$F(z_{1},\ldots,z_{m-2})=\frac{P(z_{1},\ldots,z_{m-2})}{Q(z_{1},\ldots,z_{m-2})},$$
where $P,Q \in {\mathbb C}[z_{1},\ldots,z_{m-2}]$ are realively prime polynomials.

We know that the locus of zeroes, poles, and preimages of $1$, by $F$, is given by the union of the following hyperplanes:
$$L_{j}=\{z_{j}=0\}, \; j=1,\ldots, m-2,$$
$$N_{j}=\{z_{j}=1\}, \; j=1,\ldots, m-2,$$
$$M_{ij}=\{z_{i}=z_{j}\}, \; 1 \leq i < j \leq  m-2,$$

If $P$ is non-constant, then we may write $P=\alpha P_{1}\cdots P_{r}$, where $\alpha \in {\mathbb C}\setminus\{0\}$ and $P_{j} \in {\mathbb C}[z_{1},\ldots,z_{m-2}]$ are monic irreducible polyniomials. As the zeroes of $P$ are inside the hyperplanes $L_{j}, N_{j}, M_{ij}$, then each of these irreducible polynomials must satisfy that $P_{l} \in \{z_{j}, z_{j}-1,z_{i}-z_{j}\}$.
By following similar arguments for $Q$, we may obtain that
$$F(z_{1},\ldots,z_{m-2})=\lambda \prod_{j=1}^{m-2}z_{j}^{a_{j}} \prod_{j=1}^{m-2}(z_{j}-1)^{b_{j}} \prod_{1 \leq i<j \leq m-2} (z_{i}-z_{j})^{c_{ij}},
\lambda \in {\mathbb C}\setminus\{0\}, a_{j},b_{j},c_{ij} \in {\mathbb Z}.$$

We also know that $F^{-1}(1) \subset \bigcup L_{j} \cup N_{j} \cup M_{ij}$. 

Our next claim is that $a_{j},b_{j},c_{ij} \in \{-1,0,1\}$. In fact, for example, assume that $a_{1} \notin \{-1,0,1\}$. In such a situation, the equation $F=1$ will provide a polynomial of degree at least two in the variable $z_{1}$. This will provide a solution in $\Omega_{m}$, a contradiction. The other instances are similar.

Now, we observe that 
$$\sum_{a_{j}>0}a_{j}+\sum_{b_{j}>0}b_{j}+\sum_{c_{ij}>0}c_{ij} \leq 2, \quad 
\sum_{a_{j}<0}a_{j}+\sum_{b_{j}<0}b_{j}+\sum_{c_{ij}<0}c_{ij} \geq -2.$$

Again, otherwise, the equation $F \in \{\infty,0,1\}$ will provide solutions on $\Omega_{m}$ if any one of the desired inequalities fails.

All the above asserts that $F$ has the desired form.

\subsection{Case $m,n \geq 4$}
The case $n \leq m$ is a consequence of part (1) of Proposition \ref{prop1}. Let us assume $n>m$. If $F=(F_{1},\ldots,F_{n-2}):\Omega_{m} \to \Omega_{n}$, then we consider
$\widetilde{F}=(F_{1},\ldots,F_{m-2}):\Omega_{m} \to \Omega_{m}$. By part (2) of Proposition \ref{prop1}, $\widetilde{F} \in {\mathbb G}_{m}$. Similarly, for each $j=1,\ldots,m-2$, the holomorphic map $\widetilde{F}_{j}=(F_{1},\ldots,F_{j-1},F_{m-1},F_{j+1},\ldots,F_{m-2}):\Omega_{m} \to \Omega_{m}$ also belongs to ${\mathbb G}_{m}$  (we have just replaced the coordinate $F_{j}$ by $F_{m-1}$). Now, for each $j$ there is a point $p_{j} \in \Omega_{m}$ such that $\widetilde{F}(P_{j})=\widetilde{F}_{j}(p_{j})$, that is, $F_{j}(p_{j})=F_{m-1}(p_{j})$. This asserts that $F(p_{j})$ cannot belong to $\Omega_{n}$ (it will have repeated coordinates), a contradiction.


\end{document}